\documentclass[reqno,12pt]{amsart}
\usepackage{geometry}
\usepackage{amsmath,amssymb,mathrsfs,color}
\usepackage[upint]{stix}
\usepackage[colorlinks,
linkcolor=red,
anchorcolor=green,
citecolor=blue,
]{hyperref}
\usepackage{subfigure}
\usepackage{float}
\usepackage{times}
\usepackage{tikz}
\usetikzlibrary{intersections}
\usepackage{paralist}
\usepackage{comment}

\makeatletter
\def\setliststart#1{\setcounter{\@listctr}{#1}%
  \addtocounter{\@listctr}{-1}}
\makeatother

\makeatletter
\@addtoreset{figure}{section}
\makeatother

\usepackage{calc}
 \newtheorem{The}{Theorem}[section]
 
 \newtheorem{Lem}[The]{Lemma}
 \newtheorem{Pro}[The]{Proposition}
 \theoremstyle{definition}

 \numberwithin{equation}{section}

\newcommand{\R}{\mathbb{R}}

\newcommand{\N}{\mathbb{N}}

\newcommand{\SING}{\mbox{\rm Sing}\,(u)}

\newcommand{\argmin}{\operatorname*{arg\ min}}

\newcommand{\epf}{\mbox{\rm epf}\,}

\newcommand{\gra}{\mbox{\rm graph}\,}
\newcommand{\ext}{\mbox{\rm ext}\,}

\title[Local strict singular characteristics II]{Local strict singular characteristics II: \\existence for stationary equation on $\R^2$}
\author{Wei Cheng \and Jiahui Hong}
\address{Department of Mathematics, Nanjing University, Nanjing 210093, China}
\email{chengwei@nju.edu.cn}
\address{Department of Mathematics, Nanjing University, Nanjing 210093, China}
\email{hjh9413@163.com}

\begin{document}
\maketitle

%\tableofcontents

\begin{abstract}
The notion of strict singular characteristics is important in the wellposedness issue of singular dynamics on the cut locus of the viscosity solutions. We provide an intuitive and rigorous proof of the existence of the strict singular characteristics of Hamilton-Jacobi equation $H(x,Du(x),u(x))=0$ in two dimensional case. We also proved if $\mathbf{x}$ is a strict singular characteristic, then we really have the right-differentiability of $\mathbf{x}$ and the right-continuity of $\dot{\mathbf{x}}^+(t)$ for every $t$. Such a strict singular characteristic must give a selection $p(t)\in D^+u(\mathbf{x}(t))$ such that $p(t)=\arg\min_{p\in D^+u(\mathbf{x}(t))}H(\mathbf{x}(t),p,u(\mathbf{x}(t)))$.% Higher dimensional case was also discussed.
\end{abstract}

\section{Introduction}

The singularities of the viscosity solutions of Hamilton-Jacobi equation
\begin{equation}\label{eq:intro_HJ}\tag{HJ$_s$}
	H(x,Du(x),u(x))=0,\qquad x\in\R^n,
\end{equation}
characterize the focus and crossing of associated characteristics. The occurrence of singularities of the viscosity solutions of \eqref{eq:intro_HJ} provide useful information of the complicated dynamical behavior of the underlying Hamiltonian systems as well as the invertibility of the associated Lax-Oleinik evolution from the level of functional analysis. This is closely related to many problems from mathematical physics. The readers can refer to \cite{Cannarsa_Cheng2021a} to a survey of this topic. 

This paper is devoted to some finer analysis of the propagation of singularities involving the notion of singular characteristics. If the Hamiltonian $H(x,p,u)$ is convex in $p$-variable, the semiconcavity is the so called \emph{maximal regularity} for the viscosity solutions of \eqref{eq:intro_HJ}. Throughout this paper, we suppose $H(x,p,u):\R^n\times\R^n\times\R\to\R$ is of class $C^1$ and satisfies
\begin{enumerate}[\rm (H1)]
	\item $H(x,\cdot,u)$ is strictly convex for all $(x,u)\in\R^n\times\R$.
	\item $\Omega$ is a connected open subset of $\R^n$, and $u:\Omega\to\R$ is a locally semiconcave (viscosity) solution of \eqref{eq:intro_HJ}.
\end{enumerate}

The notion of singular characteristics for Hamilton-Jacobi equations comes from early work by Albano and Cannarsa (\cite{Albano_Cannarsa2002}). The name \emph{generalized characteristics} was coined in the paper \cite{Albano_Cannarsa2002} after similar notion introduced by Dafermos \cite{Dafermos1977} in hyberbolic conservation law. Recall that a Lipschitz  arc $\mathbf{x}:[0,\tau]\to\R^n$ is called a \emph{generalized characteristic} starting from $x$ for the pair $(H,u)$ if it satisfies the following:
\begin{align}\label{intro:gc}
	\begin{cases}
		\dot{\mathbf{x}}(s)\in\mathrm{co}\,H_p\big(\mathbf{x}(s),D^+u(\mathbf{x}(s)),u(\mathbf{x}(s))\big)&\quad \text{a.e.}\;s\in[0,\tau],\\
		\dot{\mathbf{x}}(0)=x.&
	\end{cases}
\end{align}
The local structure of  singular (generalized) characteristics was further investigated  by Cannarsa and Yu in \cite{Cannarsa_Yu2009}, where \emph{singular characteristics}  were proved more regular near the starting point than the arcs constructed in  \cite{Albano_Cannarsa2002}. For any weak KAM solution $u$ of \eqref{eq:intro_HJ}, the class of  \emph{intrinsic singular characteristics} was constructed in \cite{Cannarsa_Cheng3} by Cannarsa and the first author of this paper, using the positive type Lax-Oleinik semi-group.

However, due to convex hull in the right side of \eqref{intro:gc}, there is no uniqueness result for the singular characteristics except for the well known systems of generalized gradients (\cite{ACNS2013}). Recent significant progress in the attempt to develop a more restrictive notion of singular characteristics is due to Khanin and Sobolevski (\cite{Khanin_Sobolevski2016}). In this paper, we will call such curves \emph{strict singular characteristic} but in the literature they are also refereed to as  \emph{broken characteristics}, see \cite{Stromberg2013,Stromberg_Ahmadzadeh2014}. We now proceed to recall their definition: given a semiconcave solution $u$ of  \eqref{eq:intro_HJ},  a Lipschitz singular curve $\mathbf{x}:[0,T]\to\Omega$ is called a strict singular characteristic from $x\in\SING$ if there exists a measurable selection $p(t)\in D^+u(\mathbf{x}(t))$ such that
\begin{equation}\label{eq:intro_sgc}
	\begin{split}
		\begin{cases}
		\dot{\mathbf{x}}(t)=H_p(\mathbf{x}(t),p(t),u(\mathbf{x}(t)))& a.e.\ t\in[0,T],\\
		\mathbf{x}(0)=x.&
	\end{cases}
	\end{split}
\end{equation}
As proved in \cite{Khanin_Sobolevski2016}, there exists a solution of \eqref{eq:intro_sgc} (for a time dependent Hamiltonian $H(t,x,p)$). Moreover, some additional regularity properties of such curves were established, including
\begin{itemize}[--]
	\item the right-differentiability of $\mathbf{x}$ for every $t$,
	\item the right-continuity of $\dot{\mathbf{x}}$,
	\item the fact that $p(\cdot):[0,T]\to\R^n$ satisfies $p(t)=\arg\min_{p\in D^+_xu(t,\mathbf{x}(t))}H(t,\mathbf{x}(t),p)$.
\end{itemize}
 
First general result on the uniqueness of the strict singular characteristics was obtained in \cite{Cannarsa_Cheng2021b} for the classical Tonelli-type Hamiltonian $H(x,p)$ and semiconcave solution of \eqref{eq:intro_HJ} in two dimensional case. More precisely, if the initial point $x\in\SING$ is not a critical point, i.e., $0\not\in H_p(x,D^+u(x))$, then 
\begin{itemize}[--]
	\item any two singular characteristics are identical up to a bi-Lipschitz homeomorphism,
	\item there exists a unique strict singular characteristic from $x$.
\end{itemize}

However, known proofs of the existence of strict singular characteristics still look rather mysterious (see \cite{Khanin_Sobolevski2016} and the appendix of \cite{Cannarsa_Cheng2021b}). In our previous paper \cite{Cheng_Hong2021}, we considered the problem in a more intuitive way. For a Cauchy problem
\begin{equation}\label{eq:HJe}\tag{HJ$_e$}
	\left\{
	 \begin{split}
	 	D_tu(t,x)+H(t,x,D_xu(t,x),u(t,x))=&\,0\\
	 	u(0,x)=&\,u_0(x)
	 \end{split}
	 \right.\quad x\in \R^n, t>0,
\end{equation}
with $u_0\in C^2$, we proved the local existence of smooth strict singular characteristics from non-conjugate singular point of the solution $u(t,x)$ of \eqref{eq:HJe}, when $n=1$ or $n>1$ with some degenerate initial condition. Recall that any viscosity solution of \eqref{eq:intro_HJ} or \eqref{eq:HJe} can be regarded as the value function of a Herglotz-type variational problem introduced in \cite{CCWY2019} and \cite{CCJWY2020} (see also \cite{Wang_Wang_Yan2019_1} for a representation formula by an implicit variational principle). So, one can define the conjugate point when the initial data are of class $C^2$ (see \cite{Cheng_Hong2021}), which does not work for general initial data. In this paper, we will deal with the general case mainly on $\R^2$ for the stationary equation \eqref{eq:intro_HJ}.

Since condition (H1), given $x\in\Omega$, there exists a unique $p^{\#}(x)\in D^{+}u(x)$ such that
\begin{align*}
	H(x,p^{\#}(x),u(x))\leqslant H(x,p,u(x)),\qquad \forall p\in D^{+}u(x).
\end{align*}
We denote $v^{\#}(x)=H_p(x,p^{\#}(x),u(x))$. 

Now, we can formulate our main result.

\begin{The}\label{thm:existence}
Suppose $n=2$, then for any $x_0\in\Sigma(u)$, there exists $t_0>0$ and a Lipschitz curve $\gamma:[0,t_0]\to\R^2$, $\gamma(0)=x_0$ such that
\begin{enumerate}[\rm (1)]
	\item $\dot{\gamma}^{+}(t)=v^{\#}(\gamma(t)),\ \forall t\in[0,t_0)$.
	\item $\lim_{s\to t^+}v^{\#}(\gamma(s))=v^{\#}(\gamma(t)),\ \forall t\in[0,t_0)$.
	\item $\gamma(t)\in\Sigma(u),\ \forall t\in[0,t_0]$.
\end{enumerate}
\end{The}

The paper is organized as follows. In section 2, we introduce necessary material on Hamilton-Jacobi equations, viscosity solutions and semiconcavity. In section 3, we proved the main result of this paper. There is an appendix which is composed of some technical materials on the reparameterization of Lipschitz curves. 

\medskip

\noindent\textbf{Acknowledgements.} Wei Cheng is partly supported by National Natural Science Foundation of China (Grant No. 11871267 and 11790272). 

\section{Semiconcave function and viscosity solution}

\subsection{Semiconcave function}

Let $\Omega\subset\R^n$ be a convex set. We recall that a function $u:\Omega\rightarrow\R$ is said to be {\em semiconcave} (with linear modulus) if there exists a constant $C>0$ such that
\begin{equation}\label{eq:SCC}
\lambda u(x)+(1-\lambda)u(y)-u(\lambda x+(1-\lambda)y)\leqslant\frac C2\lambda(1-\lambda)|x-y|^2
\end{equation}
for any $x,y\in\Omega$ and $\lambda\in[0,1]$. Any constant $C$ that satisfies the above inequality  is called a {\em semiconcavity constant} for $u$ in $\Omega$. A function $u:\Omega\rightarrow\R$ is said to be {\em locally semiconcave} if for each $x\in\Omega$ there exists an open ball $B(x,r)\subset\Omega$ such that $u$ is a semiconcave function on $B(x,r)$.

For any continuous function $u:\Omega\subset\R^n\to\R$ and for any $x\in\Omega$, the closed convex sets
\begin{align*}
D^-u(x)&=\left\{p\in\R^n:\liminf_{y\to x}\frac{u(y)-u(x)-\langle p,y-x\rangle}{|y-x|}\geqslant 0\right\},\\
D^+u(x)&=\left\{p\in\R^n:\limsup_{y\to x}\frac{u(y)-u(x)-\langle p,y-x\rangle}{|y-x|}\leqslant 0\right\}.
\end{align*}
are called the {\em subdifferential} and {\em superdifferential} of $u$ at $x$, respectively. If $u:\Omega\to\R$ is locally Lipschitz, a vector $p\in\R^n$ is said to be a {\em reachable} (or {\em limiting}) {\em gradient} of $u$ at $x$ if there exists a sequence $\{x_n\}\subset\Omega\setminus\{x\}$, converging to $x$, such that $u$ is differentiable at $x_k$ for each $k\in\N$ and
$$
\lim_{k\to\infty}Du(x_k)=p.
$$
The set of all reachable gradients of $u$ at $x$ is denoted by $D^{\ast}u(x)$. 

\begin{Pro}[\cite{Cannarsa_Sinestrari_book}]\label{basic_facts_of_superdifferential}
Let $u:\Omega\subset\R^n\to\R$ be a semiconcave function and let $x\in\Omega$. Then the following properties hold.
\begin{enumerate}[\rm {(}a{)}]
  \item $D^+u(x)$ is a nonempty compact convex set in $\R^n$ and $D^{\ast}u(x)\subset\partial D^+u(x)$, where  $\partial D^+u(x)$ denotes the topological boundary of $D^+u(x)$.
  \item The set-valued function $x\rightsquigarrow D^+u(x)$ is upper semicontinuous.
  \item If $D^+u(x)$ is a singleton, then $u$ is differentiable at $x$. Moreover, if $D^+u(x)$ is a singleton for every point in $\Omega$, then $u\in C^1(\Omega)$.
  \item $D^+u(x)=\mathrm{co}\, D^{\ast}u(x)$.
%  \item $D^{\ast}u(x)=\big\{\lim_{i\to\infty}p_i~:~ p_i\in D^+u(x_i),\; x_i\to x,\;\mathrm{diam}\,(D^+u(x_i))\to 0\big\}$.
\end{enumerate}
\end{Pro}

We call $\Sigma(u)$ the \emph{singular set} of a semiconcave function $u$, that is $x\in\Sigma(u)$ if $D^+u(x)$ is not a singleton.

\begin{Pro}[\cite{Cannarsa_Sinestrari_book}]
\label{criterion-Du_semiconcave2}
Let $u:\Omega\to\R$ be a continuous function. If there exists a constant $C>0$ such that, for any $x\in\Omega$, there exists $p\in\R^n$ such that
\begin{equation}\label{criterion_for_lin_semiconcave}
u(y)\leqslant u(x)+\langle p,y-x\rangle+\frac C2|y-x|^2,\quad \forall y\in\Omega,
\end{equation}
then $u$ is semiconcave with constant $C$ and $p\in D^+u(x)$.
Conversely,
if $u$ is semiconcave  in $\Omega$ with constant $C$, then \eqref{criterion_for_lin_semiconcave} holds for any $x\in\Omega$ and $p\in D^+u(x)$.
\end{Pro}

\subsection{Viscosity solution and its superdiffferential}

Since our convexity assumption (H1), instead of the standard definition, we say $u:\R^n\to\R$ is a \emph{viscosity solution} of \eqref{eq:intro_HJ} if $u$ is locally semiconcave and satisfies \eqref{eq:intro_HJ} almost everywhere. Comparing to item (d) of Proposition \ref{basic_facts_of_superdifferential}, we have the following structural result of $D^+u(x)$.

\begin{Lem}[\cite{Cannarsa_Sinestrari_book}]\label{lem:structure of D+u}
For any $x\in\R^n$, we have
\begin{align*}
	\mbox{\rm ext}\,(D^{+}u(x))=D^{*}u(x)=\{p\in D^{+}u(x):H(x,p,u(x))=0\},
\end{align*}
where $\mbox{\rm ext}\, C$ stands for the set of extremal points with respect to the convex set $C$.
\end{Lem}

Let us recall some facts from convex analysis. For any non-empty convex compact set $K\subset\R^n$ and $\theta\in\R^n$, we define
\begin{align*}
	\epf(K,\theta)=\{x\in K:\langle x,\theta \rangle \leqslant \langle y,\theta \rangle,\ \forall y\in K\}.
\end{align*}
Obviously, $\epf(K,0)=K$, and we have $\epf(K,\theta)=\epf(K,\lambda\theta)$ when $\lambda>0$.

\begin{Lem}\label{lem:epf}
Suppose $K\subset\R^n$ is a non-empty convex compact set and $\theta\in\R^n$. Then we have
\begin{enumerate}[\rm (1)]
	\item $\epf(K,\theta)$ is non-empty, convex and compact.
	\item If $x\in\epf(K,\theta)$, $x_1,x_2\in K$, $x\in(x_1,x_2)$, then $[x_1,x_2]\subset\epf(K,\theta)$.
	\item If $\epf(K,\theta)=x$ is a singleton, then $x\in\ext(K)$.
\end{enumerate}
\end{Lem}

\begin{proof}
The proof of (1) is trivial. Since $x\in(x_1,x_2)$, there exists $\lambda\in(0,1)$ such that $x=\lambda x_1+(1-\lambda)x_2$, that is, $\lambda(x-x_1)=(1-\lambda)(x_2-x)$. Recall $x\in\epf(K,\theta)$ implies $\langle x-x_1,\theta \rangle \leqslant 0$. On the other hand, $\langle x-x_1,\theta \rangle=\frac{1-\lambda}{\lambda}\langle x_2-x,\theta \rangle\geqslant 0$. Thus, we have $\langle x-x_1,\theta \rangle=0$, which implies $x_1\in\epf(K,\theta)$. Similarly, we also have $x_2\in\epf(K,\theta)$. Notice that $\epf(K,\theta)$ is convex, by \rm{(1)}. We conclude that $[x_1,x_2]\subset\epf(K,\theta)$. This completes of proof of (2). Item (3) follows directly from (2).
\end{proof}

\begin{Lem}\label{lem:convex set and function}
Suppose $K\subset\R^n$ is a non-empty convex compact set, and $f\in C^1(\R^n,\R)$ is strictly convex. Then
\begin{enumerate}[\rm (1)]
	\item There exists a unique $x\in K$ such that $f(x)\leqslant f(y)$ for all $y\in K$, and we have $x\in\epf(K,Df(x))$.
	\item If $x\in\epf(K,Df(x))$, then $f(x)\leqslant f(y)$ for all $y\in K$.
\end{enumerate}
\end{Lem}

The following Lemma is a direct consequence of Lemma  \ref{lem:structure of D+u} and Lemma \ref{lem:convex set and function}.
 
\begin{Lem}\label{lem:properties of minimal p}
For any $x\in\R^n$, we have
\begin{enumerate}[\rm (1)]
	\item There exists a unique $p^{\#}(x)\in D^{+}u(x)$ such that
	\begin{align*}
		H(x,p^{\#}(x),u(x))\leqslant H(x,p,u(x)),\qquad \forall p\in D^{+}u(x).
	\end{align*}
    \item Denote $v^{\#}(x)=H_p(x,p^{\#}(x),u(x))$, then $p^{\#}(x)\in\epf(D^{+}u(x),v^{\#}(x))$.
    \item If $p\in\epf(D^{+}u(x),H_p(x,p,u(x)))$, then $p=p^{\#}(x)$.
	\item If $x\in\Sigma(u)$, then $H(x,p^{\#}(x),u(x))<0$ and $p^{\#}(x)\in D^{+}u(x)\setminus D^{*}u(x)$.
\end{enumerate}
\end{Lem}

\section{Proof of the main result}

\subsection{A bi-Lipschitz homeomorphism}

\begin{Lem}\label{lem:translation by C2 f}
Suppose $f\in C^2(\R^n,\R)$, and let
\begin{align*}
	\tilde{H}(x,p,w)&=H(x,p+Df(x),w+f(x)),\qquad (x,p,w)\in\R^n\times\R^n\times\R,\\
	w(x)&=u(x)-f(x),\qquad x\in \R^n.
\end{align*}
Then we have
\begin{enumerate}[\rm (1)]
	\item $\tilde{H}$ is of class $C^1$ and satisfies \rm{(H1)}. $w$ is locally semiconcave on $\R^n$ and we have
	\begin{align*}
		D^+w(x)=D^+u(x)-Df(x),\ D^*w(x)=D^*u(x)-Df(x),\qquad \forall x\in\R^n.
	\end{align*}
	Moreover, $w$ is a viscosity solution of
	\begin{align*}
		\tilde{H}(x,Dw(x),w(x))=0,\qquad x\in\R^n.
	\end{align*}
    \item $\Sigma(w)=\Sigma(u)$.
	\item Let $\tilde{p},\tilde{v}$ be the vectors corresponding to $p^{\#},v^{\#}$ in Lemma \ref{lem:properties of minimal p} for $\tilde{H}$ and $w$ respectively, then
	\begin{align*}
		\tilde{p}(x)=p^{\#}(x)-Df(x),\ \tilde{v}(x)=v^{\#}(x),\qquad \forall x\in\R^n.
	\end{align*}	
\end{enumerate}
\end{Lem}

\begin{proof}
\rm{(1)} and \rm{(2)} follows from the $C^2$ regularity of $f$. We only need to prove \rm{(3)}. For any $x\in\R^n$ and $p\in D^+w(x)$, we have
\begin{align*}
	\tilde{H}(x,p^{\#}(x)-Df(x),w(x))=H(x,p^{\#}(x),u(x))\leqslant H(x,p+Df(x),u(x))=\tilde{H}(x,p,w(x)).
\end{align*}
This implies $\tilde{p}(x)=p^{\#}(x)-Df(x)$. Notice that
\begin{align*}
	\tilde{H}_p(x,p,w)=H_p(x,p+Df(x),w+f(x)),\qquad \forall (x,p,w)\in\R^n\times\R^n\times\R.
\end{align*}
Thus, we have
\begin{align*}
	\tilde{v}(x)=\tilde{H}_p(x,\tilde{p}(x),w(x))=H_p(x,\tilde{p}(x)+Df(x),w(x)+f(x))=H_p(x,p^{\#}(x),u(x))=v^{\#}(x).
\end{align*}
This completes the proof.
\end{proof}

Notice that the conclusion in Theorem \ref{thm:existence} is local. We only need to make our proof for any bounded open convex set $U\subset\R^n$ and $x_0\in U$. In the following, we suppose $U\subset\R^n$ is bounded, open and convex, and $u$ is linearly semiconcave on $U$ with constant $C_0$. Let $f(x)=\frac{1}{2}(C_0+1)|x|^2$ in Lemma \ref{lem:translation by C2 f}, then $w(x)=u(x)-\frac{1}{2}(C_0+1)|x|^2$ is uniformly concave with constant $-1$. By Lemma \ref{lem:translation by C2 f} \rm{(2)} \rm{(3)}, Theorem \ref{thm:existence} holds for $H$ and $u$ if and only if it holds for $\tilde{H}$ and $w$. Lemma \ref{lem:translation by C2 f} \rm{(1)} shows $\tilde{H}$ and $w$ satisfy the same condition as $H$ and $u$. Therefore, with out loss of generality, we always suppose that $u$ is uniformly concave on $U$ with constant $-1$.

\begin{Lem}[\cite{Cannarsa_Sinestrari_book}]\label{lem:uniformly concave function}
\hfill
\begin{enumerate}[\rm (1)]
		\item For any $x_1,x_2\in U$, $p_i\in D^{+}u(x_i),\ i=1,2$, we have
		\begin{align*}
			\langle p_2-p_1,x_2-x_1 \rangle \leqslant -|x_2-x_1|^2.
		\end{align*}
		\item Set $x\in\R^n$, $\theta\in S^{n-1}$. If $\{x_i\}\subset\R^n\setminus\{x\}$ and $p_i\in D^{+}u(x_i),\ i\in\N$ satisfies
		\begin{align*}
			\lim_{i\to\infty}x_i=x,\ \lim_{i\to\infty}\frac{x_i-x}{|x_i-x|}=\theta,\ \lim_{i\to\infty}p_i=p\in\R^n,
		\end{align*}
		then $p\in\epf(D^{+}u(x),\theta)$.
\end{enumerate}
\end{Lem}

We define the map
\begin{align*}
	\varphi:\gra(D^{+}u\Big\vert_{U})\to\R^n,\qquad (x,p)\mapsto p.
\end{align*}
Obviously, $\varphi((x,D^{+}u(x)))=D^{+}u(x)$ for all $x\in U$.

\begin{Lem}\label{lem:properties of varphi}
\hfill
\begin{enumerate}[\rm (1)]
	\item $\varphi$ is injective.
	\item $\varphi$ is Lipschitz with constant $1$.
	\item For any $x\in U$, there exists $\delta_x>0$ such that $D^{+}u(x)+B_{\delta_x}\subset \mbox{\rm Im}\,(\varphi)$.
\end{enumerate}
\end{Lem}

\begin{proof}
If $\varphi(x_i,p)=p,\ i=1,2$, then $p\in D^{+}u(x_i),\ i=1,2$. By Lemma \ref{lem:uniformly concave function} (1), we have
\begin{align*}
	0=\langle p-p,x_2-x_1 \rangle \leqslant -|x_2-x_1|^2.
\end{align*}
This implies $x_1=x_2$. Therefore, $\varphi$ is injective. This leads to (1). Item (2) follows directly from the definition of $\varphi$.

Now, turn to prove (3). For any $x\in U$, there exists $\varepsilon_x>0$ such that $\bar{B}_{\varepsilon_x}(x)\subset U$. Let $\delta_x=\frac{1}{4}\varepsilon_x$. For any $p\in D^{+}u(x)+\bar{B}_{\delta_x}$, we define
\begin{align*}
	\phi_p(y)=u(y)-\langle p,y \rangle,\qquad y\in \R^n,
\end{align*}
and denote $p'$ to be the projection of $p$ on $D^{+}u(x)$, that is, $p'=\argmin_{q\in D^{+}u(x)}|q-p|$. Then we have
\begin{align*}
	\phi_p(y)-\phi_p(x)&=u(y)-\langle p,y \rangle-u(x)+\langle p,x \rangle\\
	&=u(y)-u(x)-\langle p',y-x \rangle +\frac{1}{2}|y-x|^2+\langle p'-p,y-x \rangle-\frac{1}{2}|y-x|^2\\
	&\leqslant \langle p'-p,y-x \rangle-\frac{1}{2}|y-x|^2\\
	&\leqslant |y-x|(|p'-p|-\frac{1}{2}|y-x|)\\
	&\leqslant |y-x|(\frac{1}{4}\varepsilon_x-\frac{1}{2}|y-x|),\qquad \forall y\in\R^n.
\end{align*}
When $|y-x|=\varepsilon_x$, there holds $\phi_p(y)-\phi_p(x)\leqslant -\frac{1}{4}\varepsilon_x^2<0$. Since $\phi_p$ is also uniformly concave on $U$, $\phi_p$ has a unique maximum point $x(p)$ in $B_{\varepsilon_x}(x)$. Therefore, $0\in D^{+}\phi_p(x(p))=D^{+}u(x(p))-p$, that is, $p\in D^{+}u(x(p))$. This implies $p=\varphi(x(p),p)$. So we have $D^{+}u(x)+B_{\delta_x}\subset\mbox{\rm Im}\,(\varphi)$.
\end{proof}

Since $\varphi$ is injective, we can define
\begin{align*}
	\varphi^{-1}:\mbox{\rm Im}\,(\varphi)\to \gra(D^{+}u\Big\vert_U),\ p\mapsto\varphi^{-1}(p)=(x(p),p).
\end{align*}
Obviously, for all $p\in\mbox{\rm Im}\,(\varphi)$, we have $p\in D^{+}u(x(p))$.

\begin{Lem}\label{lem:bi-Lip}
\hfill
\begin{enumerate}[\rm (1)]
		\item $|x(p_2)-x(p_1)|\leqslant|p_2-p_1|,\ \forall p_1,p_2\in\mbox{\rm Im}\,(\varphi)$.
		\item $\varphi:\gra(D^{+}u\Big\vert_U)\to\mbox{\rm Im}\,(\varphi)$ is a bi-Lipschitz homeomorphism with Lipschitz constant $2$.
	\end{enumerate}
\end{Lem}

\begin{proof}
By Lemma \ref{lem:uniformly concave function} \rm{(1)}, we have
	\begin{align*}
		\langle p_2-p_1,x(p_2)-x(p_1) \rangle\leqslant -|x(p_2)-x(p_1)|^2.
	\end{align*}
It follows that
	\begin{align*}
		|x(p_2)-x(p_1)|^2\leqslant|p_2-p_1||x(p_2)-x(p_1)|.
	\end{align*}
If $x(p_2)=x(p_1)$, then $|x(p_2)-x(p_1)|\leqslant|p_2-p_1|$ obviously holds. If $x(p_2)\neq x(p_1)$, we also obtain $|x(p_2)-x(p_1)|\leqslant|p_2-p_1|$. (2) is a consequence of (1) and Lemma \ref{lem:properties of varphi} \rm{(2)}.
\end{proof}

Next, we denote
\begin{align*}
	A^{\#}=\{p\in\mbox{\rm Im}\,(\varphi):p\in D^+u(x(p))\setminus D^*u(x(p))\}.
\end{align*}

\begin{Lem}\label{lem:prop of A}
\hfill
\begin{enumerate}[\rm (1)]
		\item $x(A^{\#})=\Sigma(u)\cap U$.
		\item $A^{\#}$ is open in $\R^n$.
	\end{enumerate}
\end{Lem}

\begin{proof}
\rm{(1)} follows directly from Lemma \ref{lem:structure of D+u}. For (2), Lemma \ref{lem:structure of D+u} implies
\begin{align*}
	A^{\#}=\{p\in\mbox{\rm Im}\,(\varphi):H(x(p),p,u(x(p)))<0\}.
\end{align*}
By Lemma \ref{lem:properties of varphi} \rm{(3)}, $\mbox{\rm Im}\,(\varphi)$ is open in $\R^n$, and Lemma \ref{lem:bi-Lip} shows $x(\cdot)$ is Lipschitz with respect to $p$. Therefore, $A^{\#}$ is open in $\R^n$.
\end{proof}

\subsection{Proof of the main result}
In this section, we always suppose $n=2$, that is, the case on the plane.

\begin{Lem}\label{lem:R2-1}
	Suppose $x_0\in\Sigma(u)\cap U$ and $v^{\#}(x_0)\neq 0$. Then there exists $p_1,p_2\in D^*u(x_0)$ such that $\epf(D^+u(x_0),v^{\#}(x_0))=[p_1,p_2]$ and $p^{\#}(x_0)\in(p_1,p_2)$.
\end{Lem}

\begin{proof}
By Lemma \ref{lem:properties of minimal p} \rm{(2)}, we know $p^{\#}(x_0)\in\epf(D^+u(x_0),v^{\#}(x_0))$. If $\epf(D^+u(x_0),v^{\#}(x_0))=p^{\#}(x_0)$ is a singleton, Lemma \ref{lem:epf} \rm{(3)} implies $p^{\#}(x_0)\in\ext(D^+u(x_0))=D^*u(x_0)$, while Lemma \ref{lem:properties of minimal p} (4) shows $p^{\#}(x_0)\in D^+u(x_0)\setminus D^*u(x_0)$. This leads to a contradiction. So $\epf(D^+u(x_0),v^{\#}(x_0))$ is not a singleton, and we can fix $q_1,q_2\in\epf(D^+u(x_0),v^{\#}(x_0))$, $q_1\neq q_2$. It follows that $\langle q_2-q_1,v^{\#}(x_0) \rangle=0$, and for any $q\in\epf(D^+u(x_0),v^{\#}(x_0))$, we also have $\langle q-q_1,v^{\#}(x_0) \rangle=0$. Thus, there exists $\lambda\in\R$ such that $q-q_1=\lambda (q_2-q_1)$, that is, $q=q_1+\lambda(q_2-q_1)$. This implies $\epf(D^+u(x_0),v^{\#}(x_0))\subset\{q_1+\lambda(q_2-q_1):\lambda\in\R\}$, which is a line on $\R^2$. 

By Lemma \ref{lem:epf} \rm{(1)}, there exists $p_1,p_2\in D^+u(x_0)$ such that $\epf(D^+u(x_0),v^{\#}(x_0))=[p_1,p_2]$. Now we claim $p_i\in D^*u(x_0)$, $i=1,2$. In fact, if there exists $p_{1,1},p_{1,2}\in D^+u(x_0)$ such that $p_1\in(p_{1,1},p_{1,2})$, then by Lemma \ref{lem:epf} (2), we obtain $[p_{1,1},p_{1,2}]\in\epf(D^+u(x_0),v^{\#}(x_0))=[p_1,p_2]$, which leads to a contradiction. So we have $p_1\in \ext(D^+u(x_0))$, and similarly $p_2\in \ext(D^+u(x_0))$. Lemma \ref{lem:structure of D+u} shows $\ext(D^+u(x_0))=D^*u(x_0)$. It follows that $p_i\in D^*u(x_0)$, $i=1,2$. Finally, Lemma \ref{lem:properties of minimal p} \rm{(4)} implies $p^{\#}(x_0)\in(p_1,p_2)$.
\end{proof}

\begin{Lem}\label{lem:R2-2}
Under the assumption of Lemma \ref{lem:R2-1}, we suppose $p_0\in(p_1,p_2)$, $\theta_0\in S^1$ and $\langle \theta_0,v^{\#}(x_0) \rangle<0$. Let $\eta(t)=p_0+\theta_0 t$, $t\geqslant0$. Then we have
\begin{enumerate}[\rm (1)]
	\item $\lim_{t\to 0^+}\frac{x(\eta(t))-x_0}{|x(\eta(t))-x_0|}=\frac{v^{\#}(x_0)}{|v^{\#}(x_0)|}$.
	\item For any $\varepsilon>0$, there exists $t_{\varepsilon}>0$ such that
	\begin{align*}
		D^*u(x(\eta(t)))\subset\{p_1,p_2\}+B_{\varepsilon},\qquad \forall t\in(0,t_{\varepsilon}),\\
		D^*u(x(\eta(t)))\cap B_{\varepsilon}(p_i)\neq\emptyset,\ i=1,2,\qquad \forall  t\in(0,t_{\varepsilon}).
	\end{align*}
    \item $\lim_{t\to 0^+}p^{\#}(x(\eta(t)))=p^{\#}(x_0)$, $\lim_{t\to 0^+}v^{\#}(x(\eta(t)))=v^{\#}(x_0)$.
\end{enumerate}
\end{Lem}

\begin{proof}
By Lemma \ref{lem:properties of varphi} (3), there exists $t_1>0$ such that for all $t\in[0,t_1)$, there holds $\eta(t)\in\mbox{\rm Im}\,(\varphi)$, that is, $x(\eta(t))\in U$. Since $\langle \theta_0,v^{\#}(x_0) \rangle<0$, for any $t\in(0,t_1)$, we have
\begin{align*}
	\langle \eta(t),v^{\#}(x_0) \rangle=\langle p_0+\theta_0 t,v^{\#}(x_0) \rangle<\langle p_0,v^{\#}(x_0) \rangle.
\end{align*}
Noticing $p_0\in\epf(D^+u(x_0),v^{\#}(x_0))$, we know that $\eta(t)\notin D^+u(x_0)$, that is, $x(\eta(t))\neq x_0$, which implies $\frac{x(\eta(t))-x_0}{|x(\eta(t))-x_0|}\in S^1$. Now, set $\{t_i\}\subset(0,t_1)$ is a sequence satisfying
\begin{align*}
	\lim_{i\to+\infty}t_i=0,\qquad \lim_{i\to+\infty}\frac{x(\eta(t_i))-x_0}{|x(\eta(t_i))-x_0|}=\theta\in S^1.
\end{align*}
By Lemma \ref{lem:bi-Lip}, we have
\begin{align*}
	|x(\eta(t_i))-x_0|=|x(\eta(t_i))-x(\eta(0))|\leqslant|\eta(t_i)-\eta(0)|=t_i\to 0,\qquad i\to+\infty.
\end{align*}
Thus, due to Lemma \ref{lem:uniformly concave function} (2), there holds $p_0\in\epf(D^+u(x_0),\theta)$. Combing this with $p_0\in(p_1,p_2)$ and Lemma \ref{lem:epf} (2), we obtain $[p_1,p_2]\subset\epf(D^+u(x_0),\theta)$, which implies $\langle p_2-p_1,\theta \rangle=0$. Since $\langle p_2-p_1,v^{\#}(x_0) \rangle=0$ and $n=2$, we have $\theta=\pm\frac{v^{\#}(x_0)}{|v^{\#}(x_0)|}$. On the other hand, Lemma \ref{lem:uniformly concave function} (1) shows
\begin{align*}
	\langle \eta(t_i)-p_0,x(\eta(t_i))-x_0 \rangle \leqslant -|x(\eta(t_i))-x_0|^2.
\end{align*}
It follows that 
\begin{align*}
	\left\langle \theta_0,\frac{x(\eta(t_i))-x_0}{|x(\eta(t_i))-x_0|} \right\rangle\leqslant-\frac{1}{t_i}|x(\eta(t_i))-x_0|<0.
\end{align*}
By letting $i\to+\infty$, we have $\langle\theta_0,\theta\rangle\leqslant0$. Noticing $\langle \theta_0,-\frac{v^{\#}(x_0)}{|v^{\#}(x_0)|} \rangle>0$, we obtain $\theta=\frac{v^{\#}(x_0)}{|v^{\#}(x_0)|}$. The argument above shows that when $t\to0^+$, any convergent subsequence of $\frac{x(\eta(t))-x_0}{|x(\eta(t))-x_0|}$ must converge to $\frac{v^{\#}(x_0)}{|v^{\#}(x_0)|}$. Therefore, we have
\begin{align*}
	\lim_{t\to 0^+}\frac{x(\eta(t))-x_0}{|x(\eta(t))-x_0|}=\frac{v^{\#}(x_0)}{|v^{\#}(x_0)|}.
\end{align*}
This completes the proof of (1).

For the proof of (2), we divide into several steps.

\medskip

\noindent{\textbf{Step 1.}} There exists $0<t_2<t_1$ such that $D^+u(x(\eta(t)))\subset[p_1,p_2]+B_\varepsilon$ for all $t\in(0,t_2)$:\\
For any $\varepsilon>0$, there exists $0<\delta<\varepsilon$ such that
\begin{align*}
\left\langle p,\frac{v^{\#}(x_0)}{|v^{\#}(x_0)|} \right\rangle\geqslant \left\langle p_1,\frac{v^{\#}(x_0)}{|v^{\#}(x_0)|} \right\rangle+\delta,\qquad \forall p\in D^+u(x_0)\setminus([p_1,p_2]+B_{\frac{1}{2}\varepsilon}).
\end{align*}
Since $D^+u(\cdot)$ is upper-semicontinuous, there exists $0<t_{2,1}<t_1$ such that
\begin{align*}
    D^+u(x(\eta(t)))\subset D^+u(x_0)+B_{\frac{1}{2}\delta},\qquad \forall t\in(0,t_{2,1}).
\end{align*}
Now, if $t\in(0,t_{2,1})$ and $p\in D^+u(x(\eta(t)))\setminus([p_1,p_2]+B_\varepsilon)$, then
\begin{align*}
	p\in(D^+u(x_0)+B_{\frac{1}{2}\delta})\setminus([p_1,p_2]+B_{\frac{1}{2}\varepsilon}+B_{\frac{1}{2}\delta})\subset D^+u(x_0)\setminus([p_1,p_2]+B_{\frac{1}{2}\varepsilon})+B_{\frac{1}{2}\delta},
\end{align*}
which implies
\begin{equation}\label{eq:3-2 1}
	\left\langle p,\frac{v^{\#}(x_0)}{|v^{\#}(x_0)|}\right\rangle\geqslant\left\langle p_1,\frac{v^{\#}(x_0)}{|v^{\#}(x_0)|} \right\rangle+\delta-\frac{1}{2}\delta=\left\langle p_1,\frac{v^{\#}(x_0)}{|v^{\#}(x_0)|} \right\rangle+\frac{1}{2}\delta.
\end{equation}
On the other hand, for any $t\in(0,t_{2,1})$ and $p\in D^+u(x(\eta(t)))$, by Lemma \ref{lem:uniformly concave function} (2), we have $\langle p-p_1,x(\eta(t))-x_0\rangle\leqslant-|x(\eta(t))-x_0|^2$. It follows that $\langle p-p_1,\frac{x(\eta(t))-x_0}{|x(\eta(t))-x_0|}\rangle\leqslant-|x(\eta(t))-x_0|\leqslant0$,
\begin{align*}
	\left\langle p-p_1,\frac{v^{\#}(x_0)}{|v^{\#}(x_0)|}\right\rangle&=\left\langle p-p_1,\frac{x(\eta(t))-x_0}{|x(\eta(t))-x_0|}\right\rangle+\left\langle p-p_1,\frac{v^{\#}(x_0)}{|v^{\#}(x_0)|}-\frac{x(\eta(t))-x_0}{|x(\eta(t))-x_0|}\right\rangle\\
	&\leqslant\left\langle p-p_1,\frac{v^{\#}(x_0)}{|v^{\#}(x_0)|}-\frac{x(\eta(t))-x_0}{|x(\eta(t))-x_0|}\right\rangle.
\end{align*}
Notice that $|p-p_1|\leqslant \mbox{\rm diam}\,(D^+u(x_0))+\frac{1}{2}\delta$ is uniformly bounded. Due to the result in (1), there exits $0<t_2<t_{2,1}$ such that
\begin{equation}\label{eq:3-2 2}
	\left\langle p-p_1,\frac{v^{\#}(x_0)}{|v^{\#}(x_0)|}\right\rangle\leqslant \frac{1}{4}\delta,\qquad \forall t\in(0,t_2),p\in D^+u(x(\eta(t))).
\end{equation}
Combing \eqref{eq:3-2 1} and \eqref{eq:3-2 2}, we know that
\begin{align*}
	D^+u(x(\eta(t)))\subset[p_1,p_2]+B_\varepsilon,\qquad \forall t\in(0,t_2).
\end{align*}

\medskip

\noindent{\textbf{Step 2.}} There exists $0<t_3<t_2$ such that $D^*u(x(\eta(t)))\subset\{p_1,p_2\}+B_\varepsilon$ for all $t\in(0,t_3)$:\\
We denote $q_0=\frac{p_2-p_1}{|p_2-p_1|}$. By Lemma \ref{lem:structure of D+u}, we know that $[p_1+\frac{1}{2}\varepsilon q_0,p_2-\frac{1}{2}\varepsilon q_0]\subset A^{\#}$. Since $A^{\#}$ is open, there exists $\delta'<\frac{1}{2}\varepsilon$ such that $[p_1+\frac{1}{2}\varepsilon q_0,p_2-\frac{1}{2}\varepsilon q_0]+B_{\delta'}\subset A^{\#}$. By Step.1, there exists $0<t_3<t_2$ such that for all $t\in(0,t_3)$, $D^+u(x(\eta(t)))\subset[p_1,p_2]+B_{\delta'}$, which implies
\begin{align*}
	D^*u(x(\eta(t)))&\subset([p_1,p_2]+B_{\delta'})\setminus([p_1+\frac{1}{2}\varepsilon q_0,p_2-\frac{1}{2}\varepsilon q_0]+B_{\delta'})\\
	&\subset([p_1,p_1+\frac{1}{2}\varepsilon q_0)\cup(p_2-\frac{1}{2}\varepsilon q_0,p_2])+B_{\delta'}\subset\{p_1,p_2\}+B_\varepsilon.
\end{align*}

\medskip

\noindent{\textbf{Step 3.}} There exists $0<t_\varepsilon<t_3$ such that $D^*u(x(\eta(t)))\cap B_\varepsilon(p_i)\neq\emptyset$, $i=1,2$ for all $t\in(0,t_\varepsilon)$:
Without loss of generality, we assume $\varepsilon<\frac{1}{3}\mbox{\rm dist}\,(\{p_1,p_2\},p_0)$. Let $t_\varepsilon=\min\{t_3,\varepsilon\}$. Now, we have
\begin{align*}
	|\eta(t)-p_i|=|p_0+\theta_0 t-p_i|\geqslant|p_0-p_i|-t\geqslant3\varepsilon-\varepsilon=2\varepsilon,\ i=1,2,\qquad \forall t\in(0,t_\varepsilon).
\end{align*}
For any $t\in(0,t_\varepsilon)$, if $D^*u(x(\eta(t)))\cap B_\varepsilon(p_1)=\emptyset$, by Step.2,  $D^*u(x(\eta(t)))\subset B_\varepsilon(p_2)$, which implies $D^+u(x(\eta(t)))=\mbox{\rm co}\,(D^*u(x(\eta(t))))\subset B_\varepsilon(p_2)$. This leads to a contradiction with $|\eta(t)-p_2|\geqslant 2\varepsilon$. Thus, we have $D^*u(x(\eta(t)))\cap B_\varepsilon(p_1)\neq\emptyset$, and similarly $D^*u(x(\eta(t)))\cap B_\varepsilon(p_2)\neq\emptyset$. In conclusion, for all $t\in(0,t_\varepsilon)$ we have
\begin{align*}
 	D^*u(x(\eta(t)))\subset\{p_1,p_2\}+B_{\varepsilon},\qquad D^*u(x(\eta(t)))\cap B_{\varepsilon}(p_i)\neq\emptyset,\ i=1,2.
\end{align*}
This completes the proof of (2).

Finally, we turn to the proof of (3). For any $\varepsilon>0$, since $H$ is of class $C^1$ and strictly convex in $p$, there exists $\delta>0$ such that
\begin{align*}
	\min_{p\in D^+u(x_0)\setminus B_{\frac{1}{2}\varepsilon}(p^{\#}(x_0))}H(x_0,p,u(x_0))\geqslant H(x_0,p^{\#}(x_0),u(x_0))+\delta,
\end{align*}
and there exists $\varepsilon'<\frac{1}{4}\varepsilon$ such that 
\begin{equation}\label{eq:3-2 3}
H(x_0,p,u(x_0))\leqslant H(x_0,p^{\#}(x_0),u(x_0))+\frac{1}{3}\delta,\qquad \forall p\in B_{\varepsilon'}(p^{\#}(x_0)),
\end{equation}
\begin{equation}\label{eq:3-2 4}
H(x_0,p,u(x_0))\geqslant H(x_0,p^{\#}(x_0),u(x_0))+\frac{2}{3}\delta,\qquad \forall p\in(D^+u(x_0)\setminus B_{\frac{1}{2}\varepsilon}(p^{\#}(x_0)))+B_{\varepsilon'}.	
\end{equation}
By  (2), there exists $t_{\varepsilon'}>0$ such that for any $t\in(0,t_{\varepsilon'})$, we have
\begin{equation}\label{eq:3-2 5}
	D^+u(x(\eta(t)))\subset[p_1,p_2]+B_{\varepsilon'}\subset D^+u(x_0)+B_{\varepsilon'},
\end{equation}
and we can choose
\begin{align*}
	q_i(t)\in D^*u(x(\eta(t)))\cap B_{\varepsilon'}(p_i),\qquad i=1,2.
\end{align*}
Set $p^{\#}(x_0)=\lambda p_1+(1-\lambda)p_2$, where $\lambda\in(0,1)$. Then for $t\in(0,t_{\varepsilon'})$, we have
\begin{align*}
&|(\lambda q_1(t)+(1-\lambda)q_2(t))-p^{\#}(x_0)|=|(\lambda q_1(t)+(1-\lambda)q_2(t))-(\lambda q_1+(1-\lambda)q_2)|\\
\leqslant&\lambda|q_1(t)-q_1|+(1-\lambda)|q_2(t)-q_2|<\lambda\varepsilon'+(1-\lambda)\varepsilon'=\varepsilon',
\end{align*}
which implies $D^+u(x(\eta(t)))\cap B_{\varepsilon'}(p^{\#}(x_0))\neq\emptyset$. Combing this with \eqref{eq:3-2 3}, \eqref{eq:3-2 4} and \eqref{eq:3-2 5}, we know that
\begin{align*}
    p^{\#}(x(\eta(t)))&\in(D^+u(x_0)+B_{\varepsilon'})\setminus((D^+u(x_0)\setminus B_{\frac{1}{2}\varepsilon}(p^{\#}(x_0)))+B_{\varepsilon'})\\
    &\subset B_{\frac{1}{2}\varepsilon}(p^{\#}(x_0))+B_{\varepsilon'}\subset B_{\varepsilon}(p^{\#}(x_0)),\qquad \forall t\in(0,t_{\varepsilon'}).
\end{align*}
Therefore, we conclude that $\lim_{t\to 0^+}p^{\#}(x(\eta(t)))=p^{\#}(x_0)$ and
\begin{align*}
	\lim_{t\to 0^+}v^{\#}(x(\eta(t)))=\lim_{t\to 0^+}H_p(x(\eta(t)),p^{\#}(x(\eta(t))),u(x(\eta(t))))=H_p(x_0,p^{\#}(x_0),u(x_0))=v^{\#}(x_0).
\end{align*}
This completes our proof.
\end{proof}

\begin{proof}[Proof of theorem \ref{thm:existence}]
We divide the proof into several steps.

\medskip

\noindent{\textbf{Step 1.}} Set $x_0\in\Sigma(u)\cap U$. If $v^{\#}(x_0)=0$, then the curve $\gamma(t)\equiv x_0$, $t\in[0,+\infty)$ satisfies all the required properties. Next, we consider the case $v^{\#}(x_0)\neq0$. By Lemma \ref{lem:R2-1}, there exists $p_1,p_2\in D^*u(x_0)$ such that $\epf(D^+u(x_0),v^{\#}(x_0))=[p_1,p_2]$ and $p^{\#}(x_0)\in(p_1,p_2)$. Let $\eta(t)=p^{\#}(x_0)-\frac{v^{\#}(x_0)}{|v^{\#}(x_0)|}t$, $t\in\R$. Due to Lemma \ref{lem:properties of varphi} (3), there exists $t_1>0$ such that $\eta(t)\in\mbox{\rm Im}\,(\varphi)$ for all $t\in[0,t_1]$. Thus, we can define $\gamma_1(t)=x(\eta(t))$, $t\in[0,t_1]$. Obviously, $\gamma_1(0)=x_0$, and by Lemma \ref{lem:bi-Lip}, $\gamma_1=x\circ\eta$ is a Lipschitz curve. Since $p^{\#}(x_0)\in D^+u(x_0)\setminus D^*u(x_0)\subset A^{\#}$, Lemma \ref{lem:prop of A} shows there exists $0<t_2<t_1$ such that $\gamma_1(t)\in\Sigma(u)\cap U$ for all $t\in[0,t_2]$. Noticing $\langle-\frac{v^{\#}(x_0)}{|v^{\#}(x_0)|},v^{\#}(x_0) \rangle=-|v^{\#}(x_0)|<0$, by Lemma \ref{lem:R2-2} (3), there exists $0<t_3<t_2$ such that
	\begin{equation}\label{eq:3-3 1}
		\left\langle-\frac{v^{\#}(x_0)}{|v^{\#}(x_0)|},v^{\#}(\gamma_1(t)) \right\rangle\leqslant-\frac{1}{2}|v^{\#}(x_0)|,\qquad \forall t\in[0,t_3].
	\end{equation}
Using Lemma \ref{lem:R2-1} again, we know that for any $t\in[0,t_3]$, there exists $p_1(t),p_2(t)\in D^*u(\gamma_1(t))$ such that $\epf(D^+u(\gamma_1(t)),v^{\#}(\gamma_1(t)))=[p_1(t),p_2(t)]$.

\medskip

\noindent{\textbf{Step 2.}} Claim that there exists $0<t_4<t_3$ such that the curve $\gamma_1:[0,t_4]\to\R^2$ satisfies
\begin{equation}\label{eq:3-3 2}
	\eta(\tau_2(t))\in(p_1(t),p_2(t)),\qquad \forall t\in[0,t_4],\tau_2(t)<t_4,
\end{equation}
where $\tau_2(t)$ is defined in Appendix with respect to $\gamma_1:[0,t_4]\to\R^2$.

Since $p^{\#}(x_0)\in(p_1,p_2)$, we can set $\mbox{\rm dist}\,(\{p_1,p_2\},p^{\#}(x_0))=\delta_1>0$, and $\langle p_1-p^{\#}(x_0),p_2-p^{\#}(x_0)\rangle=-\delta_2<0$. It is easy to see that there exists $0<\varepsilon_1<\frac{1}{3}\delta_1$ such that
\begin{equation}\label{eq:3-3 3}
\langle q_1-q_0,q_2-q_0\rangle\leqslant-\frac{1}{2}\delta_2,\qquad \forall q_0\in B_{\varepsilon_1}(p^{\#}(x_0)), q_i\in B_{\varepsilon_1}(p_i), i=1,2.
\end{equation} 
Let $\varepsilon_2=\min\{\varepsilon_1,\frac{1}{2}|v^{\#}(x_0)|\varepsilon_1\}$. Then by Lemma \ref{lem:R2-2} (2) (3), there exists $0<t_4<t_3$ such that
\begin{align*}
p^{\#}(\gamma_1(t))\in B_{\frac{1}{3}\delta_1}(p^{\#}(x_0)),\ D^*u(\gamma_1(t))\subset\{p_1,p_2\}+B_{\varepsilon_2},\qquad \forall t\in[0,t_4].
\end{align*}
Now, consider any $t\in[0,t_4]$, $\tau_2(t)<t_4$. First, we show that $p_1(t)\in B_{\varepsilon_2}(p_1)$ and $p_2(t)\in B_{\varepsilon_2}(p_2)$. Otherwise, if $p_1(t),p_2(t)\in B_{\varepsilon_2}(p_1)$, for example. Then $p^{\#}(\gamma_1(t))\in[p_1(t),p_2(t)]\subset B_{\varepsilon_2}(p_1)$, that is, $|p^{\#}(\gamma_1(t))-p_1|<\varepsilon_2\leqslant\frac{1}{3}\delta_1$. On the other hand,
\begin{align*}
|p^{\#}(\gamma_1(t))-p_1|\geqslant |p_1-p^{\#}(x_0)|-|p^{\#}(x_0)-p^{\#}(\gamma_1(t))|\geqslant \delta_1-\frac{1}{3}\delta_1=\frac{2}{3}\delta_1.
\end{align*}
This leads to a contradiction. So we have $p_1(t)\in B_{\varepsilon_2}(p_1)$, $p_2(t)\in B_{\varepsilon_2}(p_2)$. Since $p^{\#}(x_0)\in(p_1,p_2)$, there exists $\lambda\in(0,1)$ such that $p^{\#}(x_0)=\lambda p_1+(1-\lambda)p_2$. It follows that
\begin{align*}
	\lambda p_1(t)+(1-\lambda)p_2(t)&\in(p_1(t),p_2(t))\\
	|\lambda p_1(t)+(1-\lambda)p_2(t)-p^{\#}(x_0)|&\leqslant \lambda|p_1(t)-p_1|+(1-\lambda)|p_2(t)-p_2|\leqslant \lambda\varepsilon_2+(1-\lambda)\varepsilon_2=\varepsilon_2,
\end{align*}
which implies $\mbox{\rm dist}\,\{[p_1(t),p_2(t)],p^{\#}(x_0)\}\leqslant\varepsilon_2$. Using \eqref{eq:3-3 1}, there exists a unique $t'\in\R$, $|t'|\leqslant\frac{2\varepsilon_2}{|v^{\#}(x_0)|}$ such that $\eta(t')$ is on the line containing $[p_1(t),p_2(t)]$. Notice also that
\begin{equation}\label{eq:3-3 4}
    |p_i(t)-p_i|<\varepsilon_2\leqslant\varepsilon_1,\ |\eta(t')-p^{\#}(x_0)|=|t'|\leqslant\frac{2\varepsilon_2}{|v^{\#}(x_0)|}\leqslant\varepsilon_1.
\end{equation}
Combing \eqref{eq:3-3 3} and \eqref{eq:3-3 4}, we obtain $\langle p_1(t)-\eta(t'),p_2(t)-\eta(t')\rangle\leqslant-\frac{1}{2}\delta_2<0$. So we have $\eta(t')\in(p_1(t),p_2(t))$. Finally, because $D^+u(\gamma_1(t))$ is convex and $\tau_2(t)<t_4$, there holds $\tau_2(t)=\max\{\tau\in\R:\eta(\tau)\in D^+u(\gamma_1(t))\}$. If $t'<\tau_2(t)$, then
\begin{align*}
	\langle \eta(\tau_2(t))-\eta(t'),v^{\#}(\gamma_1(t))\rangle=(\tau_2(t)-t')\left\langle-\frac{v^{\#}(x_0)}{|v^{\#}(x_0)|},v^{\#}(\gamma_1(t))\right\rangle\leqslant-\frac{1}{2}(\tau_2(t)-t')|v^{\#}(x_0)|<0.
\end{align*}
This leads to a contradiction with $\eta(t')\in[p_1(t),p_2(t)]=\epf(D^+u(\gamma_1(t)),v^{\#}(\gamma_1(t)))$. Therefore, we conclude that $t'=\tau_2(t)$, and as a result, $\eta(\tau_2(t))\in(p_1(t),p_2(t))$.

\medskip

\noindent{\textbf{Step 3.}} Since $p^{\#}(x_0)\in\epf(D^+u(x_0),v^{\#}(x_0))$, we obtain that $\gamma_1(t)\neq x_0$ for all $t>0$. Thus, we have $l(0,t_4)\geqslant \gamma_1(t_4)-x_0>0$, where $l(\cdot,\cdot)$ is defined in Appendix with respect to $\gamma_1:[0,t_4]\to\R^2$. Now, we define the curve $\gamma_2:[0,l(0,t_4)]\to\R^2$, $\gamma_2(s)=\gamma_1(t)$, where $l(0,t)=s$. Obviously, $\gamma_2(0)=\gamma_1(0)=x_0$, $\gamma_2(s)=\gamma_1(t)\in\Sigma(u)\cap U$ for all $s\in[0,l(0,t_4)]$, and by Lemma \ref{lem:v=1}, $\gamma_2$ is Lipschitz. For any $t\in[0,t_4]$, $\tau_2(t)<t_4$, using \eqref{eq:3-3 1}, \eqref{eq:3-3 2} and Lemma \ref{lem:R2-2} (1) (3), we have
\begin{align*}
	\lim_{r\to\tau_2(t)^+}\frac{\gamma_1(r)-\gamma_1(t)}{|\gamma_1(r)-\gamma_1(t)|}=\frac{v^{\#}(\gamma_1(t))}{|v^{\#}(\gamma_1(t))|},\qquad \lim_{r\to\tau_2(t)^+}v^{\#}(\gamma_1(r))=v^{\#}(\gamma_1(t)).
\end{align*}
Therefore, due to Lemma \ref{lem:length rescaling}, $\gamma_2$ satisfies
\begin{align*}
	\dot{\gamma}_2^+(s)=\frac{v^{\#}(\gamma_2(s))}{|v^{\#}(\gamma_2(s))|},\qquad \lim_{r\to s^+}v^{\#}(\gamma_2(r))=v^{\#}(\gamma_2(s)),\qquad \forall s\in[0,l(0,t_4)).
\end{align*}
Since $v^{\#}(\gamma_2(\cdot))$ is right-continuous, there exists $0<a<l(0,t_4)$ and $0<b_1\leqslant b_2<+\infty$ such that
\begin{align*}
	b_1\leqslant|v^{\#}(\gamma_2(s))|\leqslant b_2,\qquad \forall s\in[0,a].
\end{align*}
Now, by Lemma \ref{lem:ode for right d}, there exists a Lipschitz function $\tau:[0,\frac{a}{b_2}]\to[0,a]$ such that
\begin{align*}
\begin{cases}
	\dot{\tau}^+(s)=|v^{\#}(\gamma_2(\tau(s)))|,\qquad s\in[0,\frac{a}{b_2}),\\
	\tau(0)=0.
\end{cases}
\end{align*}

\medskip

\noindent{\textbf{Step 4.}} Finally, let $\gamma:[0,\frac{a}{b_2}]\to\R^2$, $\gamma(s)=\gamma_2(\tau(s))$. Obviously, $\gamma(0)=\gamma_2(0)=x_0$, $\gamma(s)\in\Sigma(u)\cap U$ for all $s\in[0,\frac{a}{b_2}]$, and $\gamma=\gamma_2\circ\tau$ is a Lipschitz curve. For any $s\in[0,\frac{a}{b_2})$, noticing that $b_1\leqslant\dot{\tau}^+(s)\leqslant b_2$ and $v^{\#}(\gamma_2(\cdot))$ is right-continuous, we have
\begin{align*}
	&\dot{\gamma}^+(s)=\dot{\gamma}_2^+(\tau(s))\cdot\dot{\tau}^+(s)=\frac{v^{\#}(\gamma_2(\tau(s)))}{|v^{\#}(\gamma_2(\tau(s)))|}\cdot|v^{\#}(\gamma_2(\tau(s)))|=v^{\#}(\gamma_2(\tau(s)))=v^{\#}(\gamma(s)),\\
	&\lim_{r\to s^+}v^{\#}(\gamma(r))=\lim_{r\to s^+}v^{\#}(\gamma_2(\tau(r)))=v^{\#}(\gamma_2(\lim_{r\to s^+}\tau(r)))=v^{\#}(\gamma_2(\tau(s)))=v^{\#}(\gamma(s)).
\end{align*}
Therefore, $\gamma:[0,\frac{a}{b_2}]\to\R^2$ satisfies all the required properties. This completes our proof.
\end{proof}

\appendix

\section{rescaling of a lipschitz curve}
In this section, we suppose $\tau_0>0$, $\gamma:[0,\tau_0]\to\R^n$ is a Lipschitz curve with Lipschitz constant $\kappa$. For $t\in[0,\tau_0]$, we define
\begin{align*}
	\tau_1(t)&=\inf\{\tau\in[0,t]:\gamma(r)=\gamma(t),\ \forall r\in[\tau,t]\},\\
	\tau_2(t)&=\sup\{\tau\in[t,\tau_0]:\gamma(r)=\gamma(t),\ \forall r\in[t,\tau]\}.
\end{align*}
Obviously, we have $0\leqslant \tau_1(t)\leqslant t\leqslant \tau_2(t)\leqslant\tau_0$. For $0\leqslant t_1\leqslant t_2\leqslant \tau_0$, we define the length of $\gamma\Big\vert_{[t_1,t_2]}$ as follows
\begin{align*}
	l(t_1,t_2)=\sup\{\sum_{i=1}^{k}|\gamma(r_i)-\gamma(r_{i-1})|:k\in\N,\ t_1=r_0\leqslant r_1\leqslant\cdots\leqslant r_k=t_2\}.
\end{align*}
We can easily obtain $0\leqslant l(t_1,t_2)\leqslant\kappa(t_2-t_1)$.

\begin{Lem}\label{lem:tau and l}
\hfill
\begin{enumerate}[\rm (1)]
	\item For any $t_1,t_2\in[0,\tau_0]$, we have
	\begin{align*}
		[\tau_1(t_1),\tau_2(t_1)]\cap[\tau_1(t_2),\tau_2(t_2)]\neq\emptyset\Longleftrightarrow[\tau_1(t_1),\tau_2(t_1)]=[\tau_1(t_2),\tau_2(t_2)]
	\end{align*}
	\item $l(t_1,t_3)=l(t_1,t_2)+l(t_2,t_3)$, $\forall 0\leqslant t_1\leqslant t_2\leqslant t_3\leqslant\tau_0$.
	\item For any $t_1,t_2\in[0,\tau_0]$, we have
	\begin{align*}
		l(0,t_1)=l(0,t_2)\Longleftrightarrow[\tau_1(t_1),\tau_2(t_1)]=[\tau_1(t_2),\tau_2(t_2)].
	\end{align*}
    \item If $t,t_i\in[0,\tau_0]$ satisfies $l(0,t_i)>l(0,t)$, $i\in\N$ and $\lim_{i\to\infty}l(0,t_i)=l(0,t)$, then
    \begin{align*}
    	t_i>\tau_2(t),\ \forall i\in\N,\qquad \lim_{i\to\infty}t_i=\tau_2(t).
    \end{align*}
\end{enumerate}
\end{Lem}

\begin{proof}
The proofs of (1) and (2) are immediate. To prove (3), without loss of generality, we assume $t_1\leqslant t_2$. Then
\begin{align*}
	l(0,t_1)=l(0,t_2) &\Longleftrightarrow l(t_1,t_2)=0 \Longleftrightarrow \gamma(t)=\gamma(t_1),\ \forall t\in[t_1,t_2]\\
	&\Longleftrightarrow t_2\in [\tau_1(t_1),\tau_2(t_1)] \Longleftrightarrow [\tau_1(t_1),\tau_2(t_1)]=[\tau_1(t_2),\tau_2(t_2)],
\end{align*}
where the first equivalence is from \rm{(2)} and the last is from \rm{(1)}.

Finally, it follows from \rm{(3)} that $l(0,\tau_2(t))=l(0,t)<l(0,t_i)$, $i\in\N$, which implies $t_i>\tau_2(t)$, $i\in\N$. If there exists $\tilde{t}>\tau_2(t)$ and a subsequence $\{t_{i_k}\}$ of $\{t_i\}$ such that $t_{i_k}\geqslant \tilde{t}$, $\forall k\in\N$, then by \rm{(3)} we have $l(t,\tilde{t})=l(0,\tilde{t})-l(0,t)>0$ and
\begin{align*}
l(0,t_{i_k})\geqslant l(0,\tilde{t})=l(0,t)+l(t,\tilde{t}),\qquad \forall k\in\N,
\end{align*}
which leads to a contradiction to $\lim_{i\to\infty}l(0,t_i)=l(0,t)$. Therefore, we have $\lim_{i\to\infty}t_i=\tau_2(t)$. This completes the proof of (4).
\end{proof}

Now, we define a curve $\gamma_l:[0,l(0,\tau_0)]\to\R^n$, $\gamma_l(s)=\gamma(t)$, where $t\in[0,\tau_0]$ satisfies $l(0,t)=s$. Due to Lemma \ref{lem:tau and l} \rm{(3)}, the curve $\gamma_l$ is well defined.

\begin{Lem}\label{lem:v=1}
$\gamma_l$ is a Lipschitz curve with constant 1 and satisfies
\begin{equation}\label{app:1}
	|\dot{\gamma}_l(s)|=1,\qquad a.e.\ s\in[0,l(0,\tau_0)].
\end{equation}
\end{Lem}

\begin{proof}
For any $0\leqslant s_1\leqslant s_2\leqslant l(0,\tau_0)$, there holds
\begin{align*}
|\gamma_l(s_2)-\gamma_l(s_1)|=|\gamma(t_2)-\gamma(t_1)|\leqslant l(t_1,t_2)=l(0,t_2)-l(0,t_1)=s_2-s_1,
\end{align*}
where $t_i\in[0,\tau_0]$ satisfies $l(0,t_i)=s_i$, $i=1,2$. Thus, $\gamma_l$ is a Lipschitz curve with constant 1 and
\begin{equation}\label{eq:5-2 1}
	|\dot{\gamma}_l(s)|\leqslant 1,\qquad a.e.\ s\in[0,l(0,\tau_0)].
\end{equation}
On the other hand, for any $0=t_0\leqslant t_1\leqslant\cdots\leqslant t_k=\tau_0$, we have
\begin{align*}
\sum_{i=1}^{k}|\gamma(t_i)-\gamma(t_{i-1})|&=\sum_{i=1}^{k}|\gamma_l(l(0,t_i))-\gamma_l(l(0,t_{i-1}))|=\sum_{i=1}^{k}|\int_{l(0,t_{i-1})}^{l(0,t_i)}\dot{\gamma}_l(s)\ ds|\\
&\leqslant \sum_{i=1}^{k}\int_{l(0,t_{i-1})}^{l(0,t_i)}|\dot{\gamma}_l(s)|\ ds=\int_{0}^{l(0,\tau_0)}|\dot{\gamma}_l(s)|\ ds.
\end{align*}
It follows that $l(0,\tau_0)\leqslant \int_{0}^{l(0,\tau_0)}|\dot{\gamma}_l(s)|\ ds$, that is, $\int_{0}^{l(0,\tau_0)}1-|\dot{\gamma}_l(s)|\ ds\leqslant 0$. Combing the inequality with \eqref{eq:5-2 1}, \eqref{app:1} follows.
\end{proof}

\begin{Lem}\label{lem:length rescaling}
Suppose $\gamma:[0,\tau_0]\to\R^n$ is a Lipschitz curve, $f:\R^n\to\R^n$ is a map.
\begin{enumerate}[\rm (1)]
	\item If for any $t\in[0,\tau_0]$ with $\tau_2(t)<\tau_0$, there holds $\lim_{r\to\tau_2(t)^+}f(\gamma(r))=f(\gamma(t))$, then we have
	\begin{align*}
		\lim_{r\to s+}f(\gamma_l(r))=f(\gamma_l(s)),\qquad \forall s\in[0,l(0,\tau_0)). 
	\end{align*}
    \item Under the assumption of \rm{(1)}, if for any $t\in[0,\tau_0]$ with $\tau_2(t)<\tau_0$, there holds 
    $$
    \lim_{r\to\tau_2(t)^+}\frac{\gamma(r)-\gamma(t)}{|\gamma(r)-\gamma(t)|}=f(\gamma(t)),
    $$
    we have
    \begin{align*}
    	\dot{\gamma}_l^+(s)=f(\gamma_l(s)),\qquad \forall s\in[0,l(0,\tau_0)).
    \end{align*}
\end{enumerate}
\end{Lem}

\begin{proof}
Set $s\in[0,l(0,\tau_0))$. For any $r_i\to s^+$, $i\to\infty$, choose $t,t_i\in[0,\tau_0]$, $i\in\N$ such that $l(0,t)=s$, $l(0,t_i)=r_i$, $i\in\N$. Lemma \ref{lem:tau and l} \rm{(4)} implies $t_i\to\tau_2(t)^+$, $i\to\infty$. Thus,
\begin{align*}
	\lim_{i\to\infty}f(\gamma_l(r_i))=\lim_{i\to\infty}f(\gamma(t_i))=f(\gamma(t))=f(\gamma_l(s)).
\end{align*}
It follows that $\lim_{r\to s+}f(\gamma_l(r))=f(\gamma_l(s))$, and (1) holds.

Now, we turn to the proof of (2). First, for almost all $s\in[0,l(0,\tau_0))$, we have
\begin{align*}
	\dot{\gamma}_l(s)=\lim_{r\to s+}\frac{\gamma_l(r)-\gamma_l(s)}{r-s}=\lim_{r\to s+}\frac{\gamma(t_r)-\gamma(t)}{|\gamma(t_r)-\gamma(t)|}\cdot\frac{|\gamma_l(r)-\gamma_l(s)|}{r-s},
\end{align*}
where $t_r$ and $t$ satisfy $l(0,t_r)=r$ and $l(0,t)=s$. Lemma \ref{lem:v=1} implies
\begin{align*}
	\lim_{r\to s+}\frac{|\gamma_l(r)-\gamma_l(s)|}{r-s}=|\dot{\gamma}_l(s)|=1.
\end{align*}
It follows from Lemma \ref{lem:tau and l} \rm{(4)} that
\begin{align*}
	\lim_{r\to s+}\frac{\gamma(t_r)-\gamma(t)}{|\gamma(t_r)-\gamma(t)|}=\lim_{r\to \tau_2(t)^+}\frac{\gamma(r)-\gamma(t)}{|\gamma(r)-\gamma(t)|}=f(\gamma(t))=f(\gamma_l(s)).
\end{align*}
Therefore, $\dot{\gamma}_l(s)=f(\gamma_l(s))$. Now, for any $s\in[0,l(0,\tau_0))$, we have
\begin{align*}
    \dot{\gamma}^+_l(s)=\lim_{r\to s+}\frac{\gamma_l(r)-\gamma_l(s)}{r-s}=\lim_{r\to s+}\frac{1}{r-s}\int_{s}^{r}\dot{\gamma}_l(\tau)\ d\tau=\lim_{r\to s+}\frac{1}{r-s}\int_{s}^{r}f(\gamma_l(\tau))\ d\tau=f(\gamma_l(s)),
\end{align*}
where the last equality is from \rm{(1)}.
\end{proof}

We say a function $g:\R\to\R$ is \emph{$b$-increasing} with $b>0$ if
\begin{align*}
	g(t_2) \geqslant g(t_1)+b(t_{2}-t_{1}),\qquad \forall t_1 \leqslant t_2.
\end{align*}

\begin{Lem}\label{lem:ode for right d}
Suppose $a>0$, $0<b_1\leqslant b_2<+\infty$, $f:[0,a]\to[b_1,b_2]$ is right-continuous. Then there exists a Lipschitz function $x:[0,\frac{a}{b_2}]\to\R$ such that it is a solution of
\begin{align*}
	\begin{cases}
		\dot{x}^+(t)=f(x(t)),\qquad t\in[0,\frac{a}{b_2})\\
		x(0)=0
	\end{cases}
\end{align*}
\end{Lem}

\begin{proof}
For $m\in\N$, we define $x_m:[0,\frac{a}{b_2}]\to\R$ as follows:
\begin{align*}
	x_m(t)=
\begin{cases}
	f(0)t,\qquad t\in[0,\frac{1}{m}\frac{a}{b_2}]\\
	x_m(\frac{i-1}{m}\frac{a}{b_2})+f(x_m(\frac{i-1}{m}\frac{a}{b_2}))(t-\frac{i-1}{m}\frac{a}{b_2}),\qquad t\in[\frac{i-1}{m}\frac{a}{b_2},\frac{i}{m}\frac{a}{b_2}],2\leqslant i\leqslant m.
\end{cases}
\end{align*}
It is easy to verify that $x_m$ are all $b_1$-increasing, $b_2$-Lipschitz functions on $[0,\frac{a}{b_2}]$, and $x_m(0)=0$. By Alzela-Ascoli theorem, there exists a subsequence $\{x_{m_k}\}$ and an $x:[0,\frac{a}{b_2}]\to\R$ such that $x_{m_k}$ converges uniformly to $x$ on $[0,\frac{a}{b_2}]$. $x$ is also $b_1$-increasing, $b_2$-Lipschitz on $[0,\frac{a}{b_2}]$, and $x(0)=0$. So we have
\begin{align*}
	0\leqslant x(t)<a,\qquad \forall t\in[0,\frac{a}{b_2}).
\end{align*}
Fix any $t\in[0,\frac{a}{b_2})$. For any $\varepsilon>0$, by the right-continuity of $f$, there exists $\delta>0$ such that
\begin{align*}
	|f(y)-f(x(t))|<\varepsilon,\qquad \forall x(t)\leqslant y\leqslant x(t)+\delta.
\end{align*}
Since $x_{m_k}$ converges uniformly to $x$, there exists $n_0\in\N$ such that $\|x_{m_k}-x\|<\frac{1}{2}\delta$ for all $k\geqslant n_0$. Thus, we have
\begin{align*}
	x_{m_k}(s)&<x(s)+\frac{1}{2}\delta\leqslant x(t)+(s-t)b_2+\frac{1}{2}\delta \leqslant x(t)+\frac{\delta}{2b_2}b_2+\frac{1}{2}\delta\\
	&=x(t)+\delta,\qquad \forall k\geqslant n_0,t\leqslant s\leqslant t+\frac{\delta}{2b_2}.
\end{align*}
Furthermore, choosing any $t<\tilde{t}<s\leqslant t+\frac{\delta}{2b_2}$, by the uniform convergence of $x_{m_k}$, there exists $n_1\in\N$ such that $\|x_{m_k}-x\|<\frac{b_1}{2}(\tilde{t}-t)$ for all $k\geqslant n_1$. Therefore, when $k\geqslant\max\{n_0,n_1\}$, $\frac{t+\tilde{t}}{2}\leqslant\tau\leqslant s$, we have
\begin{align*}
	x_{m_k}(\tau)\geqslant x_{m_k}\left(\frac{t+\tilde{t}}{2}\right)>x\left(\frac{t+\tilde{t}}{2}\right)-\frac{b_1}{2}(\tilde{t}-t)\geqslant x(t)+\frac{b_1}{2}(\tilde{t}-t)-\frac{b_1}{2}(\tilde{t}-t)=x(t),
\end{align*}
and $x_{m_k}(\tau)\leqslant x(t)+\delta$. This implies $|f(x_{m_k}(\tau))-f(x(t))|<\varepsilon$. Combing this with the definition of $x_m$, it is not hard to obtain that
\begin{align*}
	\frac{x(s)-x(\tilde{t})}{s-\tilde{t}}=\lim_{k\to\infty}\frac{x_{m_k}(s)-x_{m_k}(\tilde{t})}{s-\tilde{t}}\in(f(x(t))-\varepsilon,f(x(t))+\varepsilon).
\end{align*}
Since $\tilde{t}$ is arbitrary, it follows that
\begin{align*}
	\frac{x(s)-x(t)}{s-t}=\lim_{\tilde{t}\to t^+}\frac{x(s)-x(\tilde{t})}{s-\tilde{t}}\in(f(x(t))-\varepsilon,f(x(t))+\varepsilon).
\end{align*}
Finally, we have
\begin{align*}
	\dot{x}^+(t)=\lim_{s\to t^+}\frac{x(s)-x(t)}{s-t}=f(x(t)).
\end{align*}
This completes the proof.
\end{proof}

\bibliographystyle{plain}
\bibliography{mybib}

\end{document}